\documentclass[a4paper,11pt,reqno]{amsart}
\usepackage{amsmath,amsthm,amssymb,mathabx}
\usepackage{mathtools}
\usepackage{xcolor}

\allowdisplaybreaks

\usepackage{fullpage} 
\usepackage
[
  bookmarks = true, 
  pdfauthor = {KHughes}, 
  pdftitle = \text{Paucity and l^p-improving for curves}, 
  colorlinks = true, 
  linkcolor = blue, 
  citecolor = blue]
{hyperref}
\newcommand{\C}{\mathbb{C} }

\newcommand{\Z}{\mathbb{Z} }
\newcommand{\T}{\mathbb{T} }

\newcommand{\N}{\mathbb{N} }

\newcommand{\1}{{\bf 1}}

\newcommand{\vof}[1]{{\boldsymbol #1}}

\newcommand{\myset}{{\mathcal{X}}}
\newcommand{\avg}{{\mathcal{A}}}
\newcommand{\unavg}{\mathcal{S}}
\newcommand{\numreps}{{\mathcal{N}}} 
\newcommand{\maxnumreps}{\mathcal{M}} 

\newcommand{\IP}[2]{\langle #1,#2 \rangle}

\newcommand{\numpolys}{r}

\renewcommand{\and}{{\; \text{and} \;}}
\renewcommand{\for}{{\; \text{for} \;}}
\newcommand{\all}{{\; \text{all} \;}}

\DeclareMathOperator{\specialsols}{spe} 
\DeclareMathOperator{\genericsols}{gen} 

\newtheorem{theorem}{Theorem}[section]
\newtheorem{lemma}[theorem]{Lemma}
\newtheorem{proposition}[theorem]{Proposition}

\newtheorem{conjecture}[theorem]{Conjecture}

\theoremstyle{definition}

\theoremstyle{remark}
\newtheorem{remark}[theorem]{Remark}

\numberwithin{equation}{section}

\title{\vspace*{-1.5cm} Subcritical paucity and $\ell^p$-improving estimates for finite-type polynomial curves}
\author{\vspace{-0.5cm}Kevin Hughes}

\address{
    School of Mathematics
	\\	The University of Bristol
	\\ and the Heilbronn Insitute for Mathematical Research, Bristol, UK
}
\email{khughes.math@gmail.com}

\begin{document}

\begin{abstract}
I prove new subcritical bounds for the $\ell^p$-improving problem along restricted subsets of a degenerate curve. The key input is a new paucity estimate for associated inhomogeneous equations which is proven using an elimination method due to Wooley and Parsell--Wooley. 
\end{abstract}

\maketitle


\section{Introduction}

In this work I consider averages along a finite sequence $\myset \subseteq \N$ embedded on a polynomial curve $\gamma(T) : \Z \to \Z^\numpolys$ defined by $\numpolys$ univariate polynomials $\gamma(T) := (\phi_1(T),\phi_2(T),\dots,\phi_\numpolys(T))$. Each polynomial has integer coefficients. Assume that their degrees $\deg(\phi_j)$ for $j=1,\dots,\numpolys$ are separated in the sense that $1 \leq \deg(\phi_1) < \deg(\phi_2) < \cdots < \deg(\phi_\numpolys)$. I call such a collection of polynomials a \emph{separated system of $\numpolys$ polynomials}. The moment curve $(X,X^2,\dots,X^\numpolys)$ is a good example as well as curves of the form $(X^j,X^{j+1},\dots,X^{j+\numpolys-1})$. 
The total degree of a curve $\gamma$ is $D_\gamma := \deg{\phi_1},+\cdots+\deg{\phi_{\numpolys}}$. 

Define the \emph{truncated, forward averages along the sequence $\myset$ embedded on the curve $\gamma$} as 
\[
\avg^{\gamma}_{\myset}f(\vof{x}) 
:= 
|\myset|^{-1}\sum_{n \in \myset} f(x_1+\phi_1(n),\dots,x_k+\phi_{\numpolys}(n))
\]
for points $\vof{x} \in \Z^\numpolys$ and functions $f : \Z^\numpolys \to \C$. 
Here and throughout, I systematically use $|\myset|$ to denote the cardinality of a finite set $\myset$. I also use $|\cdot|$ to denote the Euclidean norm of a vector. It will be apparent from context which use I mean. 

Based on the standard examples - the delta function, the characteristic function of the curve $\gamma$ and the characteristic function of a large box - I conjecture the following diameter-free $\ell^p$-improving estimates. 
\begin{conjecture}\label{conjecture:strong}
Let $\gamma(T) = \big( \phi_1(T),\dots,\phi_{\numpolys}(T) \big) \subset \Z[T]$ be a separated system of $\numpolys$ polynomials with integer coefficients and total degree $D_{\gamma}$. 
If $1 \leq p \leq 2 \leq q$, then there exists a constant $C_{p,q}$ such that for all finite subsets $\myset \subset \N$, we have the inequality 
\begin{equation}\label{estimate:strong_conjecture}
\| \avg^{\gamma}_{\myset} \|_{\ell^p(\Z) \to \ell^{q}(\Z^\numpolys)} 
\leq C_{p,q} 
\left( |\myset|^{-D_{\gamma}(\frac{1}{p}-\frac{1}{q})} + |\myset|^{\frac{1}{q}-1} + |\myset|^{-\frac{1}{p}} \right) 
.\end{equation}
\end{conjecture}
\noindent 
Based on presently known results, the truth may be the following weaker `diameter dependent' estimate. 
\begin{conjecture}[Weak form]\label{conjecture:weak}
Let $\gamma(T) = \big( \phi_1(T),\dots,\phi_{\numpolys}(T) \big) \subset \Z[T]$ be a separated system of $\numpolys$ polynomials with integer coefficients and total degree $D_{\gamma}$. 
If $1 < p \leq 2 \leq q$ and $\epsilon>0$, then there exists a positive constant $C_{p,q,\epsilon}$ such that for all $\myset \subseteq \{1,\dots,N\}$, we have the inequality 
\begin{equation}\label{estimate:weak_conjecture} 
\| \avg^{\gamma}_{\myset} \|_{\ell^{p}(\Z) \to \ell^{q}(\Z^{\numpolys})} 
\leq C_{p,q,\epsilon} N^{\epsilon} 
\left( |\myset|^{-D_{\gamma}(\frac{1}{p}-\frac{1}{q})} + |\myset|^{\frac{1}{q}-1} + |\myset|^{-\frac{1}{p}} \right) 
\end{equation}
as $N \in \N$ tends to infinity. 
\end{conjecture}

When $1/q=1-1/p$, the second summand dominates the third summand and there arises a critical exponent $p_{\gamma} := 2-{D_{\gamma}}^{-1}$ in the conjecture determined by when the first two summands in \eqref{estimate:strong_conjecture} balance. 
We refer to exponents $1 \leq p < p_\gamma$ as subcritical while exponents $p_\gamma<p\leq2$ are supercritical so that $p_\gamma$ divides our analysis into subcritical and supercritical regimes. 
Moreover, \eqref{estimate:strong_conjecture} is true at the exponents $(p,q)=(1,\infty)$ and $(p,q)=(2,2)$ as shown by Young's inequality and Plancherel's theorem respectively. 
Note that my definition of subcritical vs supercritical differs from \cite{DHV} to be more in line with the use for a subcritical vs supercritical number of variables in the underlying system of Diophantine equations that I will study.

The arithmetic method of refinements in \cite{DHV} (and Section~\ref{section:AMoR}) permits us to establish a relationship between the study of inhomogeneous systems of Diophantine equations and subcritical $\ell^p$-improving estimates. 
Fix $\gamma = (\phi_1, \dots, \phi_{\numpolys})$, our separated system of $\numpolys$ univariate polynomials with integer coefficients. 
Also, fix $\myset$ to be a (possibly finite or infinite) subset of the natural numbers. 
Let $s \in \N$. For $\vof{a} \in \Z^s$ and an interval $I \subset \Z$, define the set of solutions 
\begin{equation}\label{def:numreps}
\numreps_{s}^{\gamma}(\myset,I,\vof{a}) 
:= 
\big\{ \vof{m}, \vof{n} \in (\myset\cap I)^s : \sum_{i=1}^s \big( \phi_j(n_i)-\phi_j(m_i) \big) = a_j \quad (1 \leq j \leq r) \big\} 
.\end{equation}
Furthermore, define the quantity 
\begin{align}\label{def:maxnumreps}
& \maxnumreps^\gamma_{s}(\myset,I) 
:= \sup_{\substack{\vof{a}\in\Z^{s} \\ \text{each } a_j \neq 0}} |\numreps^\gamma_{s}(\myset,I,\vof{a})|
.\end{align}
Note that the sets $\numreps^\gamma_{s}(\myset,I,\vof{a})$ are empty for $|\vof{a}|$ sufficiently large with respect to $I$, so that the supremum above is a maximum over a finite number of quantities. 

Define $[N]:=\{1,\dots,N\}$ for $N \in \N$; for our purposes, the natural numbers do not contain zero. 
The first step is to reduce the $\ell^p$-improving problem to a paucity estimate via the arithmetic method of refinements. 
\begin{theorem}\label{theorem:refinements}
Let $\myset \subseteq \N$ be a finite subset and $\gamma \subset \Z[X]$ be a separated system of $\numpolys$ polynomials, for some $\numpolys \geq 1$, which is not comprised of a single linear polynomial. 
If $s \in \N$, then for each exponent $p = 2-s^{-1}$ there exists a positive constant $C_{p,\gamma}$, depending only upon $p$ and $\gamma$, such that we have 
\begin{equation}\label{estimate:bound_avg_by_paucity}
\| \avg^{\gamma}_{\myset} \|_{\ell^{p,1}(\Z) \to \ell^{p',\infty}(\Z^{\numpolys})} 
\leq 
C_{p,\gamma}  |\myset|^{-1} 
\left( |\myset|^{s-1}  + |\myset| \times \maxnumreps^\gamma_{s-1}(\myset,I)  \right)^{\frac{1}{2s-1}}
.\end{equation}
Here, $I$ is any interval containing $\myset$. 
\end{theorem}
\noindent The exponent $p'$ is the dual exponent to $p$; that is, defined by the relation $1/p'+1/p=1$. So, in the theorem above, $p'=\frac{2s-1}{s-1}>2$ for $s \in \N$. 

In tandem with Theorem~\ref{theorem:refinements}, I conjecture the optimal paucity estimates which would imply the almost-optimal $\ell^p$-improving estimates. 
\begin{conjecture}[Paucity conjecture]\label{conjecture:paucity}
Let $\gamma$ be a separated system of polynomials with total degree $D_\gamma$. 
For each integer $1 \leq t < D_{\gamma}$ and for each $\epsilon>0$, there exists a constant $C_{\gamma,t,\epsilon}>0$, independent of $N$ such that for all $N \in \N$ we have 
\begin{equation}
\maxnumreps^\gamma_{t}(\myset,[N]) 
\leq 
C_{\gamma,t,\epsilon} |\myset\cap[N]|^{t-1+\epsilon} 
.\end{equation}
\end{conjecture}
\noindent Compare with Conjecture~3 at the bottom of page~58 in \cite{HB:notes}. 

If Conjecture~\ref{conjecture:paucity} is true, then Theorem~\ref{theorem:refinements} implies that for each $1 \leq s \leq D_{\gamma}$ we have 
\[
\| \avg^{\gamma}_{\myset\cap[N]} \|_{\ell^{\frac{2s-1}{s},1}(\Z) \to \ell^{\frac{2s-1}{s-1},\infty}(\Z^{\numpolys})} 
\lesssim_\epsilon 
N^{\epsilon} |\myset\cap[N]|^{\frac{s-1}{2s-1}-1} 
= 
N^{\epsilon} |\myset\cap[N]|^{-\frac{s}{2s-1}}
.\]
Consequently, Conjecture~\ref{conjecture:weak} would be established by interpolation with trivial bounds. 

My final step is verify Conjecture~\ref{conjecture:paucity} for $t$ sufficiently small. 
\begin{theorem}\label{theorem:paucity}
If $\myset$ is an infinite subset of $\N$ and that $\gamma$ is a separated system of $\numpolys$ polynomials, then for each $\vof{a} \in \Z^{\numpolys}\setminus\{\vof{0}\}$, we have 
\begin{equation}
\numreps^\gamma_{\numpolys}(\myset,[N],\vof{a}) 
\leq 
C_{\gamma,\numpolys,\epsilon} N^{\epsilon}  |\myset\cap[N]|^{\numpolys-1} 
\end{equation}
as $N$ tends to infinity. 
The constant does not depend on $\vof{a}$. 
\end{theorem}

As an immediate corollary, we obtain the bound 
\begin{equation}
\maxnumreps^\gamma_{\numpolys}(\myset,[N]) 
\leq 
C_{\gamma,\numpolys,\epsilon} N^{\epsilon}  |\myset\cap[N]|^{\numpolys-1} 
\end{equation}
as $N$ tends to infinity. 
Using this latter estimate, Theorem~\ref{theorem:refinements} and Theorem~\ref{theorem:paucity} combine to immediately imply the following $\ell^p$-improving estimate. 
\begin{theorem}\label{theorem:improving}
Suppose that $\myset$ is an infinite subset of $\N$ and that $\gamma$ is a separated system of $\numpolys$ polynomials. 
Then \eqref{estimate:weak_conjecture} of Conjecture~\ref{conjecture:weak} holds for $p:=2-\frac{1}{r+1}$ and $q:=p'$. 
\end{theorem}
\noindent 
The moment curve $(X,X^2,\dots,X^\numpolys)$ of any degree $\numpolys$ provides an interesting example of this theorem. Previously, the same bounds were known only for $r=2$ or 3. Similarly, I obtain the same bounds for curves of the form $(X^j,X^{j+1},\dots,X^{j+\numpolys-1})$; these bounds are new for all $j,r \geq 2$ - even for $(X^2,X^3)$.

\begin{remark}
By an elaboration of my methods, I can obtain similar $\ell^{2,1}(\Z) \to \ell^{q,\infty}(\Z^{\numpolys})$ estimates for certain $q \geq 2$. 
I leave this to the interested reader. 
\end{remark}

In collaborations, I will use Wooley's elimination method to give new subcritical discrete restriction estimate and new square function estimates. 

\subsection{Historical development}

The method of refinements is due to Michael Christ in \cite{Christ}. This method has been highly influential in the theory of continuous $L^p$-improving estimates. I do not attempt to survey this large body of works. 
Instead, I will mention that, to the best of my knowledge, this method has been used in the study of discrete fractional integrals and discrete $\ell^p$-improving estimates in \cite{Oberlin:discrete, Pierce:Heisenberg, Kim, DHV}. 

A crucial difference in \cite{DHV} compared to earlier works is the novel use of the method of refinements to prune special subvarieties from the analysis and reduce the problem to proving paucity estimates for inhomogeneous systems of equations. In turn, \cite{DHV} used Wooley's second order differencing method from \cite{Wooley:simultaneous, Wooley:symmetric, HW} to prove new paucity estimates for inhomogeneous systems of equations. Altogether, this led to new estimates of strength \eqref{estimate:weak_conjecture} for the $\ell^p$-improving problem in the following cases: 
\begin{itemize}
\item $\gamma(T)=(\phi(T))$ where $\phi$ is a single polynomial of degree at least two, 
\item $\gamma(T)=(T,\phi(T))$ $\phi$ is a single polynomial of degree at least two, and 
\item $\gamma(T)=(T,T^2,T^3)$ is the twisted cubic, alternately known as the moment curve in three dimensions. 
\end{itemize}
Previous results in the subcritical range were known when $\gamma(T)=(T^2)$ and $\gamma(T)=(T,\phi_2(T))$ where $\phi_2(T)$ is a quadratic polynomial. Indeed, Conjecture~\ref{conjecture:weak} was proved in these cases; see \cite{HKLMY} for more information.

\subsection{Approach}

My approach follows the one introduced in \cite{DHV}. 
In particular, Conjecture~\ref{conjecture:weak} is a direct generalization of Conjecture~1 in \cite{DHV} and Theorem~\ref{theorem:improving} is a generalization of Theorem~1 in \cite{DHV}. 
I adopt a more modular approach that explicitly states Theorem~\ref{theorem:refinements}. Theorem~\ref{theorem:refinements} was known to the authors in \cite{DHV}, but was suppressed from the exposition therein. Naturally, the proof of Theorem~\ref{theorem:refinements} in Section~\ref{section:AMoR} uses Christ's method of refinements introduced by Michael Christ in\cite{Christ}. 
However, I endeavored to give a novel exposition which is - hopefully - friendlier to analytic number theorists. In particular, my presentation is less encumbered by an imposing tower of refinements and emphasizes how to prune special subvarieties from the analysis of $\ell^p$-improving inequalities. 

The main new ingredient in my work is the use of Wooley's elimination method. This allows me to prove Theorem~\ref{theorem:paucity} which was out of reach of Wooley's second order differencing method used in \cite{DHV}. See \cite{Wooley:symmetric, PW} for previous works using elimination theory in the study of Diophantine equations.

\subsection{Outline of the paper}

In the next section, I give another exposition of the arithmetic method of refinements and prove Theorem~\ref{theorem:refinements}. 
In the final section, I prove Theorem~\ref{theorem:paucity} and discuss where the proof could be shortened an simplified for use in proving Theorem~\ref{theorem:improving}.

\subsection*{Acknowledgements}

I thank Trevor Wooley for suggesting to look at his works \cite{PW, Wooley:symmetric, Wooley:simultaneous} in relation to another problem; these provided the critical insight in extending the results of \cite{DHV} to those herein. 
I thank Marco Vitturi and Spyros Dendrinos for their collaboration in \cite{DHV}; my understanding of and perspective on the method of refinements developed through working with them. In particular, I formulated Theorem~\ref{theorem:refinements} during collaboration with them on \cite{DHV} though we chose not to include it in there. Any errors here are my own. 
I thank Julia Brandes and Oscar Marmon for conversations on the topic of inohomogeneous systems of Diophantine equations. 

Part of the work was undertaken while I participated in the Harmonic Analysis and Analytic Number Theory Trimester Program at the Hausdorff Institute of Mathematics in the summer of 2021. Their hospitality is greatly appreciated.

%
%
%
%

\section{The arithmetic method of refinements and the reduction to paucity estimates}\label{section:AMoR}

In this section I give our arithmetic refinement method motivated by the (continuous) method introduced in \cite{Christ}. 
A key ingredient in the method of refinements is Christ's `flowing lemma'. 
We give an abstract presentation of the flowing lemma for a linear operator $T$ whose adjoint operator is $T^*$. Eftsoons we return to the averaging operators and begin the method refinements. My approach consists of two steps where the first one alone is insufficient for my purposes. The second one is a better version, suited to removing special subvarieties from our analysis. This latter functionality was the novel aspect of the arithmetic method of refinements in \cite{DHV}. 
While there is no new functionality in my approach here, it is my hope that a different view of the method and full details will open up the area and entice number theorists to it.

\subsection{Flowing lemma} 

By Lorentz theory, the operator bound $\| T \|_{\ell^{p,1} \to \ell^{p',\infty}} \leq C$, for an operator $T$, is equivalent to the estimates 
\begin{equation}\label{estimate:restricted_weak}
\IP{T\1_{E}}{\1_F} 
\leq 
C |E|^{\frac{1}{p}} |F|^{\frac{1}{p}} 
\end{equation}
for all finite, non-empty subsets of the integers $E$ and $F$. 
When $p := 2-\frac{1}{s}$ estimate \eqref{estimate:restricted_weak} holds if and only if 
\[
\IP{T\1_{E}}{\1_F}^{2s-1} 
\leq 
C^{2s-1} |E|^{s} |F|^{s}.
\] 
Soon we will take $T = |\myset|\avg_{\myset}$ and show that $C^{2s-1} \lesssim_{\gamma,s} |\myset|^{s-1} + \maxnumreps^\gamma(s-1;\myset) \times |\myset|$. 
Similarly, the Lorentz space bound $\| T \|_{\ell^{2,1} \to \ell^{q,\infty}} \leq C$, for an operator $T$, is equivalent to the estimates 
\begin{equation}
\IP{T\1_{E}}{\1_F} 
\leq 
C |E|^{\frac{1}{2}} |F|^{\frac{1}{q'}} 
\end{equation}
for all finite non-empty subsets $E,F$ of the integers. 
When $p=2$ and $q=\frac{2s}{s-1}$, this becomes 
\[
\IP{T\1_{E}}{\1_F}^{2s} 
\leq 
C^{2s} |E|^{s} |F|^{s+1}
.\]

The following lemma, due to Christ in \cite{Christ}, is the foundation of the method of refinements. 
\begin{lemma}[Flowing lemma]\label{lemma:flowing}
Suppose that $T$ is a linear operator acting on characteristic functions of sets in $\Z^k$ which is bounded on $\ell^2(\Z^k)$ with a well-defined adjoint $T^*$. 
Let $E_0$ and $F_0$ be two finite subsets of $\Z^k$. 
For each $J \in \N$, there exists two sequences of subsets $E_0 \supset E_1 \supset \cdots \supset E_J$ and $F_0 \supset F_1 \supset \cdots \supset F_J$ satisfying the properties that for each $j \in [J]$ we have 
\begin{equation}\label{estimate:left_mean}
T\1_{E_{j}}(\vof{x}) \geq 2^{-j} \frac{\IP{T\1_{E_0}}{\1_{F_0}}}{|F_0|} 
\for \vof{x} \in F_j 
\and 
\end{equation}
\begin{equation}\label{estimate:right_mean}
T^*\1_{F_{j-1}}(\vof{y}) \geq 2^{-j} \frac{\IP{T\1_{E_0}}{\1_{F_0}}}{|E_0|} 
\for \vof{y} \in E_j 
.\end{equation}
Moreover, the sets $E_j$ and $F_j$ are large in the sense that 
\begin{equation}\label{estimate:right_density}
\IP{T\1_{E_{j}}}{\1_{F_{j}}} 
\geq 
2^{-j} {\IP{T\1_{E_0}}{\1_{F_0}}}
.\end{equation}
\end{lemma}

\begin{proof}
The argument is pigeonholing and induction. 
We iteratively define our sets $E_j, F_j$ with the base case $j=0$ being the initial sets $E_0, F_0$ in the statement of our theorem. 
By induction, having defined $E_j$, define the next pair of sets $E_{j+1}, F_{j+1}$ successively as 
\begin{align*}
& E_{j+1} := \left\{ 
\vof{y} \in E_j : 
T^*\1_{F_j}(\vof{y}) \geq 2^{-(j+1)} \frac{\IP{T\1_{E_0}}{\1_{F_0}}}{|E_0|} 
\right\}
\and \\ & 
F_{j+1} := \left\{ 
\vof{x} \in F_j : 
T^*\1_{E_{j+1}}(\vof{x}) \geq 2^{-(j+1)} \frac{\IP{T\1_{E_0}}{\1_{F_0}}}{|E_0|} 
\right\}
.\end{align*}
The bounds \eqref{estimate:left_mean} and  \eqref{estimate:right_mean} are satisfied by definition of the sets. 
It remains to show that the sets intersect well enough in the sense of \eqref{estimate:right_density}. 

The estimate \eqref{estimate:right_density} is clearly true for $j=0$; this serves as our base case for an induction. 
Assume by induction that \eqref{estimate:right_density} holds for some $j \in [J]$; we show that it holds for $j+1$ as well. 
First observe that our induction hypothesis implies that 
\begin{align*}
\IP{T\1_{E_{j}}}{\1_{F_{j} \setminus F_{j+1}}} 
= 
\sum_{\vof{x} \in F_{j} \setminus F_{j+1}} T\1_{E_{j}}(\vof{x}) 
< 
2^{-j} \frac{\IP{T\1_{E_0}}{\1_{F_0}}}{|F_0|} \cdot |F_{j} \setminus F_{j+1}| 
< 
2^{-j} \IP{T\1_{E_0}}{\1_{F_0}} 
.\end{align*}
This implies that 
\[
\IP{\1_{E_{j}}}{T^*\1_{F_{j+1}}} 
= 
\IP{T\1_{E_{j}}}{\1_{F_{j+1}}} 
\geq 
2^{-j} \IP{T\1_{E_0}}{\1_{F_0}} 
,\]
which is not quite what we want. 
We apply the same analysis to the other side of the inner product now: 
\begin{align*}
\IP{\1_{E_{j} \setminus E_{j+1}}}{T^*\1_{F_{j+1}}} 
= 
\hspace{-3mm}\sum_{\vof{y} \in E_{j} \setminus E_{j+1}} \hspace{-3mm}T^*\1_{F_{j+1}}(\vof{y}) 
< 
2^{-(j+1)} \frac{\IP{T\1_{E_0}}{\1_{F_0}}}{|E_0|} |E_{j} \setminus E_{j+1}| 
< 
2^{-(j+1)} \IP{T\1_{E_0}}{\1_{F_0}} 
.\end{align*}
This implies \eqref{estimate:right_density} as desired. 
\end{proof}

Now let us interpret the above discussion for our averages. For this I recall our setting. 
Fix a set $\myset \subseteq [N]$ where $N\in\N$. 
Let $\gamma = ( \phi_1,\dots,\phi_{\numpolys} ) \subset \Z[X]$ be a finite, separated system of polynomials of total degree $D_{\gamma}$. 
Fixing our curve $\gamma$, we drop our dependence of it in the notation. 
In our arguments I find it easier to work with the un-normalized operators 
\[
\unavg_{\myset}f(\vof{x}) 
:= 
|\myset| \cdot \avg_{\myset}f(\vof{x}) 
= 
\sum_{n\in\myset} f(x_1+\phi_1(n),\dots,x_{\numpolys}+\phi_{\numpolys}(n)).
\]
We will also work with the adjoint of $\unavg_{\myset}$ which is 
\[
\unavg^{*}_{\myset}f(\vof{x}) 
:= 
\sum_{n\in\myset} f(x_1-\phi_1(n),\dots,x_{\numpolys}-\phi_{\numpolys}(n))
\]
Thus, $(\unavg^{\gamma}_{\myset})^* = \unavg^{-\gamma}_{\myset}$. 
In other words, 
\begin{align*}
& \unavg_{\myset}\1_{E}(\vof{x}) 
= 
|\{ n \in \myset : \vof{x}+\gamma(n) \in E \}| 
\; \and \;
\\ & 
\unavg^{*}_{\myset}\1_{F}(\vof{x}) 
= 
|\{ n \in \myset : \vof{x}-\gamma(n) \in F \}| 
\end{align*}
for any $\vof{x} \in \Z^k$ and any subsets $E,F$ of $\Z^k$. 
We reinterpret Lemma~\ref{lemma:flowing} by introducing the means 
\begin{align*}
& \alpha := |F|^{-1} \IP{\unavg_{\myset}\1_E}{\1_F} 
\and 
\\ & \beta := |E|^{-1} \IP{\unavg_{\myset}\1_E}{\1_F} = |E|^{-1} \IP{\1_E}{\unavg^{*}_{\myset}\1_F} 
\end{align*}
so that, by \eqref{estimate:restricted_weak}, \eqref{estimate:bound_avg_by_paucity} is equivalent to proving that 
\begin{equation}\label{estimate:restricted_weak:unnormalized}
\alpha^{s} \beta^{s-1} 
\lesssim_{\gamma,s} 
\left[ |\myset|^{s-1} + |\myset| \cdot \maxnumreps^{\gamma}_{s-1}(\myset,[N]) \right] |E|
\end{equation}
for any finite subset $E$ in $\Z^k$. 

Now, let us make a few important observations before initiating the method of refinements. 
Define the measure 
\[ 
\mu^{\gamma}_{\myset} 
:= 
\sum_{n\in\myset} \delta_{\gamma(n)} 
.\]
This is not quite the same as $\1_{\{\gamma(n):n\in\myset\}}$ because the curve may take the same value more than once. However, they are comparable since the curve can only intersect itself finitely many time. 
This measure is related to our averages by  $\unavg^{\gamma}_{\myset} f = \mu^{\gamma}_{\myset} \star f$ where $\star$ denotes the convolution of two functions on $\Z^{\numpolys}$. 
I will use the following two bounds
\begin{align*}
\| \mu^{\gamma}_{\myset} \|_{\ell^1} = |\myset|
\quad \and \quad 
\| \mu^{\gamma}_{\myset} \|_{\ell^\infty} \leq D_\gamma
.\end{align*}

Observe that 
\(
0 \leq \alpha, \beta \leq |\myset|
\)
since 
\[
\|\unavg_{\myset}\|_{\ell^\infty \to \ell^\infty} = \|\unavg^{*}_{\myset}\|_{\ell^\infty \to \ell^\infty} = \| \mu^{\gamma}_{\myset} \|_{\ell^1} = |\myset|
.\] 
Hence, it suffices to assume that $\alpha \geq C$ for some constant $C>0$ depending only on the curve $\gamma$ and on $s$. 
Otherwise, we deduce that $\alpha^s \beta^{s-1} \leq C^s \times |\myset|^{s-1} \lesssim |\myset|^{s-1} \leq |\myset|^{s-1} |E|$ as desired in \eqref{estimate:restricted_weak:unnormalized}. 
Here, we used the discrete property that $|E| \geq 1$ for any nonempty subset $E$ of $\Z^k$. 
Similarly, we have 
\[
\alpha 
\leq \| \unavg_{\myset}\1_{E} \|_{\ell^\infty} 
\leq \| \mu^{\gamma}_{\myset} \|_{\ell^\infty} \cdot \| \1_{E} \|_{\ell^1} 
\leq D_{\gamma} |E|,
\]
and if $\beta \leq C$ for some positive constant $C$, then $\alpha^s \beta^s \leq C^{s-1} D_{\gamma} |\myset|^{s-1} |E|$ so that \eqref{estimate:restricted_weak:unnormalized} is satisfied. 

Let $K_\gamma = \deg{\phi_1} \cdots \deg{\phi_\numpolys}$. 
Henceforth, we choose our constant $C$ to be $2^{s+1}(s+1)K_{\gamma}$ so that we are assuming $\alpha, \beta \geq 2^{s+1}(s+1)K_{\gamma}$. 
Since $\alpha \leq |\myset|$, it suffices to prove that 
\begin{equation}\label{estimate:restricted_weak:unnormalized:large_means}
\alpha^{s-1} \beta^{s-1} 
\lesssim_{\gamma,s} 
\maxnumreps^{\gamma}_{s-1}(\myset,[N]) |E| 
\quad \text{provided that} \quad 
\alpha,\beta \geq 2^{s+1}(s+1)K_{\gamma}. 
\end{equation}
Moreover, since $\alpha$ is positive, \eqref{estimate:right_density} is positive and therefore the sets $E_s$ and $F_s$ have positive measure. 
In particular, $E_s$ and $F_s$ are non-empty.

\subsection{The initial tower of parameters}

From Lemma~\ref{lemma:flowing} we have sets $E_s \subset \dots \subset E_1 \subset E_0 := E$ and $F_s \subset \dots \subset F_1 \subset F_0 := F$ such that 
\begin{align*}
& \unavg_{\myset}\1_{E_{j}}(\vof{x}) \geq \alpha/2^j 
\for \vof{x} \in F_j, 
\and
\\ 
& \unavg^{*}_{\myset}\1_{F_{j-1}}(\vof{y}) \geq \beta/2^j 
\for \vof{y} \in E_j 
\end{align*}
for each $j \in [s]$. 
Since $\alpha,\beta > 0 $, we have that $E_s$ and $F_s$ are non-empty and therefore each set contains at least one element. 
Let $\vof{y}$ be a fixed element of $E_s$ and define the sets 
\begin{align*}
& B_1 
:= 
\{ n_1 \in \myset : \vof{y}-\gamma(n_1) \in F_{s-1} \}
\and \\ & 
A_{1}
:= 
\{ (m_1,n_1) \in \myset \times B_1 : \vof{y}-\gamma(n_1)+\gamma(m_1) \in E_{s-1} \} 
.\end{align*} 
We suppress the element $\vof{y}$ from the notation in these sets because the particular choice of $\vof{y}$ is not relevant in our analysis; merely its existence is sufficient. 
By \eqref{estimate:right_mean} we have that $|B_1| \geq \beta/2^s$ and by \eqref{estimate:left_mean} we have that 
\[
| \{ m_1 \in \myset : \vof{y}-\gamma(n_1)+\gamma(m_1) \in E_{s-1} \} | 
\geq \alpha/2^s
\]
for each $n_1 \in B_1$. 
Therefore, 
\(
|A_1|
\geq (\alpha/2^{s}) (\beta/2^{s}) 
= \alpha \beta / 2^{2s} 
.\)

Having defined the sets $A_1$ and $B_1$, we inductively define $B_t$ and $A_t$ for $t>1$ as follows. 
Assume that $B_t$ and $A_t$ are defined and write $\vof{m} = (m_1,\dots,m_t)$ and $\vof{n} = (n_1,\dots,n_t)$; define 
\begin{align*}
B_{t+1}
& := 
\big\{ (\vof{m};\vof{n},n_{t+1}) \in A_t \times \myset : \vof{y}- \sum_{i=1}^{t+1} \gamma(n_i)+\sum_{j=1}^t \gamma(m_j) \in F_{s-t-1} 
\big\} 
\and \\ 
A_{t+1} 
& := 
\big\{ (\vof{m},m_{t+1};\vof{n},n_{t+1}) \in \myset^{t+1} \times \myset^{t+1} : \vof{y}-\sum_{i=1}^{t+1} \gamma(n_i)+\sum_{j=1}^{t+1} \gamma(m_j)  \in E_{s-t-1} 
\\ & \qquad \quad 
\and 
(\vof{m};\vof{n},n_{t+1}) \in B_{t+1}
\big\} 
.\end{align*}
Similar to $A_1$ and $B_1$, we have that 
\[
|A_t| 
\geq 
2^{-2(s+s-1+\cdots+t)} (\alpha \beta)^t 
\quad \for \quad 
t \in [s]
.\]

At this stage bounding $|A_s|$ from above would yield an upper bound for $(\alpha \beta)^s$, which is our goal. 
However, such a bound on the number of solutions to the underlying system of equations is too weak to the yield sharp $\ell^p$-improving bounds that we desire. 
Let us examine this approach in the case $s=2$, where we flow twice back and forth, and $\gamma = (X^2)$. 

For each $y \in E_2$, the above discussion says that 
\[
A_2 
:= 
|\{ (m_1,n_1,m_2,n_2) \in \myset^4 : y+\phi(m_1)-\phi(n_1)+\phi(m_2)-\phi(n_2) \in E \}| 
\geq 
\alpha^2 \beta^2 /64.
\]
Unfortunately, for each point $y \in E_2$, we have that 
\(
A_2 
\geq |\myset|^{2}
\)
which does not satisfy \eqref{estimate:restricted_weak:unnormalized} when, say, $|\myset|^{1/3} \leq \beta, |E| < 2|\myset|^{1/3}$. 
To justify our claim, take $y \in E_2$. 
Note that we also have $y \in E$. 
Consequently, the set 
\[
\{ (m_1,n_1,m_2,n_2) \in \myset^4 : y+\phi(m_1)-\phi(n_1)+\phi(m_2)-\phi(n_2) = y \}
\]
is a subset of the one above defining $A_2$, and it has cardinality at least $|\myset|^2$ by considering the diagonal solutions $m_1=n_1$ and $m_2=n_2$ in $\myset$. 

\subsection{The pruned tower of parameters}

In order to overcome the difficulty discussed in the previous section, I will prune the sets $A_t$ and $B_t$ to avoid diagonal solutions and other special solutions. At each stage of the pruning, an essential feature is that the remaining `generic' solutions outweigh the pruned `special' solutions. This pruning potentially allows for sharper $\ell^p$-improving estimates. 
There is added flexibility in this approach then used below. For instance, one can add more conditions beyond $\phi(m_{t+1}) \neq \phi(n_j)$ for all $j\in[t]$ such as assuming that $|\phi(m_i)|$ and $|\phi(n_i)|$ are at least size $C$ for some fixed finite constant and all $i\in[t]$, or $\vof{m}$ does not lie in some subvariety of positive codimension.

We revisit the base case in our construction of the tower of sets $B_1, A_1, B_2, A_2, \dots, B_s, A_s$ and partition the set $A_1$ into \emph{generic} solutions $A_1^{\genericsols}$ and \emph{special} solutions $A_1^{\specialsols}$ by defining these pieces as 
\begin{align*}
& 
A_{1}^{\genericsols}
:= 
\{ (m_1,n_1) \in A_1 : \phi(m_1) \neq  \phi(n_1) \for \all \phi \in \gamma \} 
\and 
\\ & 
A_{1}^{\specialsols} 
:= 
\{ (m_1,n_1) \in A_1 : \phi(m_1) = \phi(n_1) \for \text{ some } \phi \in \gamma \} 
.\end{align*}
We will not split the set $B_1$ into any components. 
Therefore, $|B_1| \geq \beta/2^s$ while 
\[ 
|A_1| 
= |A_1^{\genericsols}|+|A_1^{\specialsols}| 
\geq 2^{-s}\alpha |B_1| 
\geq 2^{-2s}\alpha \beta
.\]
Making use of our assumption that $\alpha,\beta \geq 2^{s+1}(s+1)K_{\gamma} = 2^3 K_{\gamma}$,  I will show that $|A_1^{\genericsols}| \geq |A_1^{\specialsols}|$. 
To see this observe that for each $\phi \in \gamma$ the Fundamental Theorem of Algebra implies that 
\[
| \{ m_1 \in \Z : \phi(m_1) = \phi(n_1) \} | \leq \deg(\phi)
\]
holds for each fixed $n_1 \in \myset$. 
Applying this estimate across all $\phi \in \gamma$, we deduce that 
\[
| \{ m_1 \in \myset : 
\phi(m_1) = \phi(n_1) \for\text{some } \phi \in \gamma \} | 
\leq \big( \prod_{\phi\in\gamma} \deg{\phi} \big)
= K_{\gamma}
\]
for each $n_1 \in \myset$. 
Consequently, $|A_1^{\specialsols}| \leq K_{\gamma} |B_1|$ which implies that 
$|A_1^{\genericsols}| \geq (2^{3-1}-1)|A_1^{\specialsols}|$ and 
\( |A_1^{\genericsols}| \geq \alpha \beta / 2^{2s+1}. \)

We now conclude the proof of Theorem~\ref{theorem:refinements} for the case $s=2$, and will return to $s>2$ in a moment. 
Assume that $s=2$. 
Our discussion above gives sets $B_1, A_1^{\genericsols}, A_1^{\specialsols}$ as defined above with the property that $|A_1^{\genericsols}| \geq \alpha \beta /2^5$. 
Our claim here is that $|A_1^{\genericsols}| \leq \maxnumreps^{\gamma}_{1}(\myset,[N]) |E|$ from which Theorem~\ref{theorem:refinements} immediately follows. 
By the union bound it suffices to show that 
\[
| \{ (m_1,n_1) \in A_1^{\genericsols} : \vof{y}-\gamma(n_1)+\gamma(m_2) = \vof{z} \} | 
\leq 
\maxnumreps^{\gamma}_{1}(\myset,[N]) 
\]
for each $\vof{z} \in E$. 
The key fact here is that our restriction $\phi(m_1) \neq  \phi(n_1)$ for $(m_1,n_1) \in A_1^{\genericsols}$ and $\phi \in \gamma$ \emph{prohibits} ${y}_i = {z}_i$ for each coordinate. 
Our claim now follows from the definition \eqref{def:maxnumreps} of $\maxnumreps^{\gamma}_{1}(\myset,[N])$. 

We return to the general case and now assume that $s>2$. 
Having defined the sets $B_1$, $A_1^{\genericsols}$ and $A_1^{\specialsols}$ above, we inductively define $B_t^{\genericsols}$, $B_t^{\specialsols}$, $A_t^{\genericsols}$ and $A_t^{\specialsols}$ for $1<t<s$ as follows. 
First, take $B_1^{\genericsols} = B_1$ so that $B_1^{\specialsols}$ is the empty set. 
Assume that $t>1$ and $B_t^{\genericsols}$, $B_t^{\specialsols}$, $A_t^{\genericsols}$ and $A_t^{\specialsols}$ are defined, and write $\vof{m} = (m_1,\dots,m_t)$ and $\vof{n} = (n_1,\dots,n_t)$. 
Define the generic sets 
\begin{align*}
B_{t+1}^{\genericsols}
& := 
\big\{ (\vof{m};\vof{n},n_{t+1}) \in B_t : (\vof{m};\vof{n}) \in A_t^{\genericsols} \and  \phi(n_{t+1}) \neq  \phi(m_{j}) 
\for j \in [t] \and \phi \in \gamma
\big\}, 
\\  
A_{t+1}^{\genericsols} 
& := 
\big\{ (\vof{m},m_{t+1};\vof{n},n_{t+1}) \in A_{t+1} : (\vof{m};\vof{n},n_{t+1}) \in B_{t+1}^{\genericsols}, 
\phi(m_{t+1}) \neq  \phi(n_{j}) 
\for j \in [t] \and \phi \in \gamma
\\ & \qquad \quad 
\and 
\gamma(m_1)+\dots+\gamma(m_{t+1}) \neq   \gamma(n_1)+\dots+\gamma(n_{t+1})
\big\} 
.\end{align*}
We define the special solutions as the complement of the generic ones; that is, 
\begin{align*}
B_{t+1}^{\specialsols}
& := 
\big\{ (\vof{m};\vof{n},n_{t+1}) \in B_t : (\vof{m};\vof{n}) \in A_t^{\genericsols} 
\big\} \setminus B_{t+1}^{\genericsols}, 
\\  
A_{t+1}^{\specialsols} 
& := 
\big\{ (\vof{m},m_{t+1};\vof{n},n_{t+1}) \in A_{t+1} : (\vof{m};\vof{n},n_{t+1}) \in B_{t+1}^{\genericsols}
\big\} \setminus A_{t+1}^{\genericsols} 
.\end{align*}
%

At each stage we chose $A_t^{\genericsols} \subseteq A_t$ and $B_t^{\genericsols} \subset B_t$. 
Furthermore, the Flowing Lemma applies so that 
\begin{align}
& |A_t^{\genericsols}| + |A_t^{\specialsols}| 
\geq \alpha |B_t^{\genericsols}| / 2^t \label{inequality:A}
\and 
\\ & 
|B_{t+1}^{\genericsols}| + |B_{t+1}^{\specialsols}| \label{inequality:B} 
\geq \beta |A_t^{\genericsols}| / 2^t
\end{align}
for $t \in [s]$. 
The union bound implies 
\begin{equation}
|A_t^{\genericsols}| 
\leq \maxnumreps^{\gamma}_{t}(\myset,[N]) \times |E_{s-t}|
.\end{equation}
Therefore, the proof of our theorem is finished if we show that $|A_t^{\genericsols}| \gtrsim (\alpha \beta)^t$ for $t \in [s]$. 
In other words, it remains to show that $A_t^{\genericsols}$ is a sufficiently large subset of $A_t$. 
The inequalities \eqref{inequality:A} and \eqref{inequality:B} imply that it suffices to show that $|A_t^{\genericsols}| \geq |A_t^{\specialsols}|$ and $|B_t^{\genericsols}| \geq |B_t^{\specialsols}|$. 
Using our restriction $\alpha, \beta \geq 2^{s+1}(s+1)K_{\gamma}$, this will be an immediate consequence of the following proposition which says that our special sets are sufficiently small. 
\begin{proposition}
For each $t \in \N$ we have the bounds 
\begin{align}
& |B_{t+1}^{\specialsols}| 
\leq t K_{\gamma} |A_{t}^{\genericsols}| 
\quad \and \quad 
\\ & 
|A_{t}^{\specialsols}| 
\leq (t+1) K_{\gamma} |B_{t}^{\genericsols}| 
.\end{align}
\end{proposition}

\begin{proof}
Fix $t \in \N$. 
I start with the first inequality. 
For each $(\vof{m};\vof{n},n_{t}) \in B_t^{\specialsols}$, I show that there are at most $t D$ points in $A_t^{\genericsols}$ giving rise to $(\vof{m};\vof{n},n_{t})$. 
Here, I am writing $(\vof{m}, \vof{n})$ for points in $A_{t-1}^{\genericsols}$; each are vectors with $t-1$ coordinates. 
If $t=1$, then these points do not exist and I simply mean $(n_1)$. 
By definition, $(\vof{m};\vof{n})$ is in $A_t^{\genericsols}$, so it suffices to bound the possible number of $n_t \in \myset$ such that $\phi(n_t) = \phi(m_j)$ for some $j \in [t]$ and $\phi \in \gamma$. 
By the Fundamental Theorem of Algebra, for each $\phi \in \gamma$, there are at most $\deg(\phi)$ values, say $x$ such that $\phi(x) = y$ for each $y \in \Z$. 
This gives the first inequality. 

The bound for the second inequality follows similarly. 
The one caveat is that we have an extra possible condition to consider when $\gamma(m_1)+\dots+\gamma(m_{t}) =  \gamma(n_1)+\dots+\gamma(n_{t})$. 
Hence, $t+1$ in place of $t$. 
\end{proof}

\section{The paucity estimates}\label{section:paucity}

We define the function 
\[ 
\mathcal{L}(c,X) := \exp\bigg(c\frac{\log{X}}{\log{\log{X}}}\bigg)
.\]
By taking logarithms it is easy to see that $\mathcal{L}(c,X)$ is sub-polynomial, meaning that for all $\epsilon>0$, there exists a constant $C_{c}(\epsilon)$ depending on $c$ and $\epsilon$ such that 
\begin{equation}
\mathcal{L}(c,X) 
\leq C_{c}(\epsilon) X^\epsilon
\end{equation}
for all sufficiently large, positive $X$. 
The function $\mathcal{L}(c,X)$ naturally arises in our argument because we will need to control the number of divisors at various times. For this, we use the classical divisor bound. 
\begin{proposition}[Divisor bound]
For each $s\in\N$, there exists a constant $c_{s}$ such that if $M\in\N$, then 
\begin{equation}
|\{d_1,\dots,d_s \in \N : d_1 \cdots d_s = M\}| 
\leq \mathcal{L}(c_{s},M)
.\end{equation}
\end{proposition}

Fix $\gamma(T) := (\phi_1(T),\dots,\phi_{\numpolys}(T))$ to be a separated polynomial curve. 
In order to prove Theorem~\ref{theorem:paucity}, it suffices to show that there exists a positive constant $c_\gamma$, depending on the curve $\gamma$, such that for each $\vof{a} \neq \vof{0}$ we have the estimate 
\begin{equation}\label{estimate:paucity}
\numreps_{\numpolys}^{\gamma}(\myset,[X],\vof{a})| 
\lesssim_{\gamma} 
\mathcal{L}(c_\gamma,X) |\myset \cap [X]|^{r-1}
.\end{equation}
Each estimate will make use of elimination theory from \cite{Wooley:symmetric} and \cite{PW}. I recall what we need from these works in the next subsection.

\subsection{Elimination theory}


For $\numpolys \in \N$, define the Vandermonde determinant 
\[ 
V_\numpolys(X_1,\dots,X_\numpolys) 
:= 
\det(X_j^{i-1})_{i,j\in[\numpolys]}
= 
\prod_{1 \leq i < j \leq \numpolys} (X_j-X_i)
.\]
Call a polynomial $F(X_1,\dots,X_\numpolys) \in \Z[X_1,\dots,X_\numpolys]$ \emph{asymptotically definite} if there exists a $\lambda>0$ such that $|\vof{x}| \geq \lambda$ implies that $|F(\vof{x})| \geq 1$. For such an asymptotically definite polynomial, there are only $\lesssim \lambda^r$ possible solutions to $F(X_1,\dots,X_{\numpolys}) = 0$, and hence, only finitely many solutions. 

\begin{lemma}[Lemma~1 of \cite{PW}]\label{lemma:PW1}
Suppose that $\gamma:=(\phi_1,\dots,\phi_\numpolys)$ is a separated, well-conditioned curve. 
There exists an asymptotically definite polynomial $P_\gamma(X_1,\dots,X_\numpolys)$, of total degree 
\( 
D_\gamma 
= \sum_{i=1}^\numpolys (k_i - i) 
,\) 
such that 
\begin{equation}
\det(\phi_i'(x_j))_{i,j\in[\numpolys]} 
= 
V_\numpolys(X_1,\dots,X_\numpolys) \cdot P_\gamma(X_1,\dots,X_\numpolys)
.\end{equation}
\end{lemma}

\begin{lemma}[Lemma~2 of \cite{PW}]\label{lemma:PW2}
There exists a (non-zero) polynomial $Q_\gamma(T_1,\dots,T_\numpolys) \in \Z[T_1,\dots,T_\numpolys]$, of degree 
\( 
D_{\gamma,\numpolys} \geq 1
,\) 
such that 
\begin{equation}\label{property:eliminant1}
Q\big( \sigma_{1,\numpolys-1}(\vof{X}'),\dots,\sigma_{\numpolys,\numpolys-1}(\vof{X}') \big)
\equiv 0
,\end{equation}
but 
\begin{equation}\label{property:eliminant2}
Q\big( \sigma_{1,\numpolys}(\vof{X}),\dots,\sigma_{\numpolys,\numpolys}(\vof{X}) \big)
\not\equiv 0
.\end{equation}
\end{lemma}
\noindent In short, the existence of a polynomial $Q$ satisfying \eqref{property:eliminant1} follows from the fact that $\sigma_{i,\numpolys-1}$ has transcendence degree $\numpolys-1$ which is less than $\numpolys$. The second property \eqref{property:eliminant2} follows from considering $Q$ of minimal degree and showing that if \eqref{property:eliminant2} did not hold then one of the partial derivatives of $Q$ satisfied the hypotheses of the lemma, thereby contradicting the minimal property of $Q$. 

\begin{lemma}[Pages~467-468 of \cite{PW}]\label{lemma:PW3}
There exists a positive constant $C_{\gamma,\numpolys}$ such that for any integers $M_1,\dots,M_\numpolys$ there are at most $C_{\gamma,\numpolys}$ solutions $(n_1,\cdots,n_{\numpolys})$ in $\N^{\numpolys}$ to the equations 
\begin{equation}
M_i = \sum_{j=1}^{\numpolys} \phi_i(n_j)
\quad \for \quad 1 \leq i \leq \numpolys
.\end{equation}
\end{lemma}

\noindent In short, this lemma is proved by breaking the putative solutions into singular solutions and nonsingular solutions. Bezout's Theorem guarantees that there are at most $k_1 \cdots k_\numpolys$ possible nonsingular solutions. Meanwhile, the singular solutions correspond to those $\vof{x}\in\Z^{\numpolys}$ such that  $\det(\phi_i'(x_j))_{i,j\in[\numpolys]} = 0$. Lemma~\ref{lemma:PW1} implies that there are finitely many since $P$ is asymptotically definite. 

For $s \in \N$ and $i\in[\numpolys]$, define the multivariate polynomials 
\[
\sigma_{i,s}(X_1,\dots,X_s) 
:= 
\sum_{j=1}^{s} \phi_i(X_j)
.\]
For each $i \in [\numpolys]$, we have the following factors 
\[
(X_i-Y) | P_{\gamma} \big( \sigma_{1,r}(\vof{X})-\phi_1(Y), \dots, \sigma_{r,r}(\vof{X})-\phi_{\numpolys}(Y) \big)
.\]
This is seen by taking $X_i=Y$ and applying Lemma~\ref{lemma:PW2}. 
Define $R(X_1,\dots,X_r;Y)$ as the quotient polynomial 
\[
P_{\gamma} \big( \sigma_{1,r}(\vof{X})-\phi_1(Y), \dots, \sigma_{r,r}(\vof{X})-\phi_{\numpolys}(Y) \big) 
= 
R(X_1,\dots,X_r;Y) \prod_{i=1}^{r} (X_i-Y) 
.\]

\subsection{Proof of Theorem~\ref{theorem:paucity}}

Recall that we wish to prove \eqref{estimate:paucity}. Fix $\myset$ an infinite subset of $\N$ and $X \in \N$. Also, fix $\vof{a} \in \Z^{\numpolys}\setminus\{\vof{0}\}$. 
The proof will use induction. 
There are two base cases for our induction. When $\myset=\N$, these base cases are considered and \eqref{estimate:paucity} is proven in \cite{DHV}. The general case of $\myset \subset \N$, an infinite subset, is similarly proved. I include the details for completeness. 

\emph{The base cases:} 
The first base case is the equation 
\[
\phi_1(m)-\phi_1(n) = a_1 \neq 0
\]
where $\phi_1$ has degree at least two. 
There exists a bivariate polynomial $\chi(X,Y)$, the first order differencing polynomial, such that 
\[
(X-Y)\chi(X,Y) = \phi_1(X)-\phi_1(Y)
.\] 
Thus, $m-n$ divides $a_1$; call this factor $d_1$ and call the remaining factor $d_0:=\chi(m,n)$ so that $d_0 d_1 = a_1$. Sine $a_1$ is non-zero, so must $d_0$ and $d_1$ be non-zero. By the Divisor Bound, there are at most $\mathcal{L}(c_{1,\phi_1},X)$ possibilities for $(d_0,d_1)$. The dependence on $\phi_1$ in $c$ stems from the fact that the image of $[X]$ under $\phi$ grows like $\{x \in \Z : |x| \lesssim X^{\deg{\phi_1}}\}$ rather than a range independent of $\phi_1$. 
Fixing any such possibility $(d_0,d_1)$ and substituting the linear equation into the first order difference polynomial, the Fundamental Theorem of Algebra implies that there are at most $\deg{\phi_1}$ possibilities for $(m,n)$. 

The second base case is the pair of equations 
\begin{align*}
\phi_1(m_1)+\phi_1(m_2)-\phi_1(n_1)-\phi_1(n_2) = a_1
\\
\phi_2(m_1)+\phi_2(m_2)-\phi_2(n_1)-\phi_2(n_2) = a_2
\end{align*}
where $(a_1,a_2) \neq (0,0)$ and $\deg(\phi_1)=1$. 
Let us reduce to the case where $\phi_1(X):=X$. 
If $\phi_1(X)=\alpha X+\beta$ for some $\alpha,\beta \in \Z$, then we may, without any loss of generality, replace $\phi_1(X)$ by $\alpha X$. I abuse notation and redefine $\phi_1(X):=\alpha X$. 
Similarly, we may pull out the factor of $\alpha$ to reduce the system of equations 
Assume first that 
\begin{align*}
m_1+m_2-n_1-n_2 = a_1/\alpha
\\
\phi_2(m_1)+\phi_2(m_2)-\phi_2(n_1)-\phi_2(n_2) = a_2
.\end{align*}
We may further suppose that $\alpha$ divides $a_1$; otherwise, there are no integral solutions to the linear equation and the bound \eqref{estimate:paucity} trivially holds. At this point, I abuse notation and replace $a_1/\alpha$ by $a_1$. 

Let us consider the diagonal solutions $\{m_1,m_2\}=\{n_1,n_2\}$. Suppose that $m_2=n_2$. There are $|\myset\cap[X]|$ possibilities for $(m_2,n_2)$ under this assumption. 
We have the system of equations 
\[
m_1-n_1=a_1 
\quad \text{and} \quad 
\phi_2(m_1)-\phi_2(n_1)=a_2
.\]
Substituting the linear equation into the non-linear one, we find that 
\[
\phi_2(n_1+a_1)-\phi_2(n_1)=a_2
.\]
the Fundamental Theorem of Algebra says that there are at most $\deg{\phi_2}$ possible choices for $n_1$ provided that $\phi_2(X+a_1)-\phi_2(X)-a_2$ is not the 0 polynomial. This can only happen if $a_1=0$ in which case $m_1=n_1$ and $m_2=n_2$. The latter implies that $a_2=0$, but we have a contradiction because $(a_1,a_2) \neq (0,0)$. Therefore, there are at most $\deg{\phi_2} |\myset\cap[X]|$ possibilities for $(m_1,m_2,n_1,n_2)$ when $m_2=n_2$. 
Similarly, there are at most $\deg{\phi_2} |\myset\cap[X]|$ possibilities when $m_1=n_2$, $m_1=n_1$ or $m_1=n_2$. 
In sum, there are at most $4\deg{\phi_2} |\myset\cap[X]|$ possible diagonal solutions. 
Observe that in the arithmetic method of refinements we could avoid this case. 

Now we assume that $\{m_1,m_2\} \neq \{n_1,n_2\}$. 
Using the linear equation, write $m_1=n_1+n_2-m_2+a_1$ and substitute this into the non-linear equation to find that 
\[
\phi_2(n_1+n_2-m_2+a_1)+\phi_2(m_2)-\phi_2(n_1)-\phi_2(n_2) 
= a_2
.\]
By the Binomial Theorem, there exists a polynomial $\rho(X,Y) \in \Z[X,Y]$ such that 
\[
\phi_2(X+Y) = \phi_2(X)+\rho(X,Y)
.\]
Therefore, 
\[
\phi_2(n_1+n_2-m_2)+\phi_2(m_2)-\phi_2(n_1)-\phi_2(n_2) 
= a_2-\rho(n_1+n_2-m_2,a_1) 
.\]
Using the second order differencing polynomial $\psi(X,Y,Z)$ defined by the relation 
\[
\phi_2(X+Y-Z)+\phi_2(Z)-\phi_2(X)-\phi_2(Y) 
= 
(X-Z)(Y-Z)\psi(X,Y,Z)
,\]
we have 
\[
(n_1-m_2)(n_2-m_2)\psi(n_1,n_2,m_2)
= a_2-\rho(n_1+n_2-m_2,a_1) 
.\]
We split our analysis into two cases as to whether $a_2-\rho(n_1+n_2-m_2,a_1)$ is zero or not. 

In the first case, assume that $a_2-\rho(n_1+n_2-m_2,a_1)$ is zero. 
There are at most $\deg{\rho}=\deg{\phi}$ possible roots to the polynomial equation $a_2-\rho(Y,a_1)=0$ since $a_1$ is fixed; call these roots $y_1,\dots,y_s$ where $s \leq \deg{\phi}$. Then $n_1+n_2-m_2=y_l$ for some $l\in[s]$. Since we are in the off-diagonal situation, we find that $\psi(n_1,n_2,m_2)=0$. Substituting in our new linear equation, we deduce that $\psi(n_1,n_2,n_1+n_2-y_l)=0$. 
Lemma~3.5 of \cite{HW} implies that there are at most  $(\deg{\psi}) |\myset\cap[X]|$ possible solutions, provided that $\psi(X,Y,X+Y-y_l)$ is not the zero-polynomial - this must be the case. Otherwise, if $\psi(X,Y,X+Y-y_l)$ is the zero-polynomial, then 
\[
\phi_2(y_l)+\phi_2(X+Y-y_l)-\phi_2(X)-\phi_2(Y) 
\equiv 0
\]
which is impossible since $\phi_2$ has degree at least two. 
Observe that in the arithmetic method of refinements we could avoid this case.

In the second case, assume that $a_2-\rho(n_1+n_2-m_2,a_1)$ is not zero. Fix the value $\nu:=n_1+n_2-m_1$. There are at most $|\myset\cap[X]|$ possibilities for $\nu$, since  $\nu=m_1-a_1$.\footnote{Without the observation that $\nu=m_1-a_1$, we would instead obtain estimates in terms of the cardinality of the sum-and-difference set $\myset+\myset-\myset$. This can be significantly larger than the cardinality of the set $\myset$.} 
With $a_1, a_2$ and $\nu$ fixed, we see that  $M:=a_2-\rho(\nu,a_1)$ is also fixed. There are at most $\mathcal{L}(c_{3,\gamma},X)$ possible triples $(d_0,d_1,d_2) \in \N^3$ such that $d_0 d_1 d_2 = M$. Hence, there are this many possibilities that 
\[
n_1-m_1=d_1, 
n_2-m_1=d_2 
\text{ and }
\psi(n_1,n_2,m_1)=d_0
.\]
Substituting the linear equations into $\psi$ and using the Fundamental Theorem of Algebra, we see that there are at most $\deg{\psi}=\deg{\phi_2}-2$ possibilities for $m_1$ from which $n_1$ and $n_2$ are uniquely determined since $d_1$ and $d_2$ are fixed. 
In sum, there are at most $(\deg{\phi_2}-2)\mathcal{L}(c_{3,\gamma},X)|\myset\cap[X]|$ possibilities for the second case. 

Combining the two cases, we see that we have at most \eqref{estimate:paucity} altogether. 
This completes the proof of the base cases. 

\emph{The inductive step:} 
Suppose that Theorem~\ref{theorem:paucity} is true for all $s<r$ where $r \geq 2$ and finite. 
Let $\gamma:=(\phi_1,\dots,\phi_{\numpolys})$ be a fixed separated system of $r$ polynomials. 
We break our solutions $\numreps^{\gamma}_{\numpolys}(\myset,[X],\vof{a})$ into three sets. 
The first set are those solutions such that $m_i=n_j$ for some pair $i,j\in[\numpolys]$; call these solutions $\numreps_{s}^{\gamma,1}(\myset,X,\vof{a})$. 
The second set of solutions are those for which $P_{\gamma}(\sigma_{1,r-1}(n_2,\dots,n_r)+a_1,\dots,\sigma_{r,r-1}(n_2,\dots,n_r)+a_r)=0$; call these solutions $\numreps_{s}^{\gamma,2}(\myset,X,\vof{a})$. The third set of solutions are those remaining; call these solutions $\numreps_{s}^{\gamma,3}(\myset,X,\vof{a})$. 
We treat these cases in turn. Each case satisfies the desired bound so that there sum does as well. 
Observe that the first two cases, $\numreps_{s}^{\gamma,1}(\myset,X,\vof{a})$ and $\numreps_{s}^{\gamma,2}(\myset,X,\vof{a})$ could be removed by the arithmetic method of refinements so that only the third case needs to be treated to establish the $\ell^p$-improving estimate of Theorem~\ref{theorem:improving}. 

{\emph{Case} $\numreps_{s}^{\gamma,1}(\myset,X,\vof{a})$:} 
Assume that there exists a pair $i,j\in[\numpolys]$ such that $m_i=n_j$. Then we are reduced to the case of the lemma where $s=r-1$ with the polynomials $(\phi_2,\phi_3,\dots,\phi_{\numpolys})$ and $(a_2,a_3,\dots,a_r)$ or the polynomials $(\phi_1,\phi_3,\dots,\phi_{\numpolys})$ and $(a_1,a_3,\dots,a_r)$. At least, one of $(a_1,a_3,\dots,a_r)$ or $(a_2,a_3,\dots,a_r)$ is not zero. Observe that the separated assumption implies that either $\phi_1$ or $\phi_2$ has degree at least two. 
Thus, the base cases apply, and by induction, there are at most $\lesssim \mathcal{L}(c_{\numpolys-1,\gamma},X) |\myset\cap[X]|^{\numpolys-2}$ possibilities which, when combined with the $|\myset\cap[X]|$ possibilities for $m_i=n_j$, is acceptable compared to \eqref{estimate:paucity}. 

{\emph{Case} $\numreps_{s}^{\gamma,2}(\myset,X,\vof{a})$:} 
Define 
\begin{align*}
M_i 
= \sum_{j=1}^r \phi_i(m_j) 
= a_i + \sum_{j=1}^r \phi_i(n_j) 
\quad \text{for} \quad i\in[\numpolys]
.\end{align*}
By Lemma~3.5 of \cite{HW}, there are at most $(\deg{P})k_r|\myset\cap[X]|^{r-1}$ possible $\vof{M}$ such that $P(\vof{M})=0$ since the degree of $P(\sigma_{1,r}(\vof{X}),\dots,\sigma_{r,r}(\vof{X}))$ is at most $\deg{P} \cdot \max_{i\in[\numpolys]}\{\deg{\phi_i}\}$. 
Lemma~\ref{lemma:PW3} says that there are at most $\lesssim 1$ possible $\vof{m} \in \N^\numpolys$. Applying Lemma~\ref{lemma:PW3} again, but this time to 
\begin{align*}
M_i-a_i
=\sum_{j=1}^r \phi_i(n_j) 
\quad \text{for} \quad i\in[\numpolys]
,\end{align*}
we deduce that there are $\lesssim 1$ possible $\vof{n} \in \N^\numpolys$. Therefore, there are $\lesssim 1$ pairs $(\vof{m},\vof{n}) \in \N^{\numpolys} \times \N^{\numpolys}$ for each $\vof{M}$. And in sum, there are at most $\lesssim_{\gamma,\numpolys} |\myset\cap[X]|^{r-1}$ possible $(\vof{m},\vof{n})$ in this case. 

{\emph{Case} $\numreps_{s}^{\gamma,3}(\myset,X,\vof{a})$:} 
Fix $n_2,\dots,n_r \in \myset$. This fixes $M_i:=a_i+\phi_i(n_2)+\cdots+\phi_i(n_r)$ for all $i\in[\numpolys]$. In turn, $P(M_1,\dots,M_r)$ is fixed. 
In this case we have $P(\vof{M}) \neq 0$ which implies $1 \leq |P(\vof{M})| \lesssim_{\gamma} X^{\deg{P}}$ where the implicit constant depends on the coefficients of the polynomials. There are at most $\mathcal{L}(c_{\numpolys+1,\gamma},X)$, for some positive $c_{\numpolys+1,\gamma}$, possible $\numpolys+1$-tuples $(d_0,d_1,\dots,d_\numpolys)\in \N^{\numpolys+1}$ such that $d_0 d_1 \cdots d_\numpolys = P(\vof{M})$. Set $d_i:=m_i-n_1$ for $i \in [\numpolys]$ and $d_0:=P(\vof{M})/(d_1\cdots d_{\numpolys})$. Then $m_i=d_i+n_1$. Inserting these linear equations into the quotient polynomial $R(X_1,\dots,X_\numpolys;Y)$, we find that $R(d_1+n_1,\dots,d_\numpolys+n_1;n_1)=d_0$ is a univariate polynomial in $n_1$. Therefore, there are at most $\deg{R}$ possibilities for $n_1$. Each possibility for $n_1$ uniquely determines $m_i$ since $d_i$ is fixed. In sum, this gives at most $(\deg{R}) \, \mathcal{L}(c_{\numpolys+1,\gamma},X) |\myset\cap[X]|^{\numpolys-1}$ possibilities. 

\subsection{Remark: Paucity for $\numreps^{\gamma}_{\numpolys}(\myset,[X],\vof{0})$}

Lemma~\ref{lemma:PW3} implies that 
$\numreps^{\gamma}_{\numpolys}(\myset,[X],\vof{0}) \lesssim |\myset\cap[X]|^{\numpolys}$ as $X$ tends to infinity. 
This bound implies an $O(1)$-bound at a subcritical index for the related $\ell^{2,1}(\Z) \to L^{2\numpolys,\infty}(\T^\numpolys)$ discrete restricted inequality. This is the first such result observed for degenerate curves of finite type. 
I point out that the above argument allows us to prove the following paucity result for $\numreps^{\gamma}_{\numpolys}(\myset,[X],\vof{0})$. 

Define the diagonal solutions $\mathcal{D}_{s}^{\gamma}(\myset,[X])$ in $(\vof{m};\vof{n}) \in \numreps_{s}^{\gamma}(\myset,[X],\vof{0})$ as those such that 
\[
\{ \phi_i(m_1),\dots,\phi_i(m_s) \} \neq  \{ \phi_i(n_1),\dots,\phi_i(n_s) \} 
\quad \for \quad 1 \leq i \leq \numpolys
.\]
{This condition is stronger than assuming that $(m_1,\dots,m_s)$ is {not} a permutation of $(n_1,\dots,n_s)$ since a polynomial may have multiple solutions to a specific value.} 
For each $s\in\N$, there are between $|\myset \cap [X]|^s$ and $C_\gamma |\myset \cap [X]|^s$ diagonal solutions for some positive constant $C_\gamma$. 
\begin{theorem}
Suppose that $\myset$ is an infinite subset of $\N$. Then as $X$ tends to infinity, we have the bound 
\begin{equation}
|\numreps_{\numpolys}^{\gamma}(\myset,[X],\vof{0}) \setminus \mathcal{D}_{s}^{\gamma}(\myset,[X])|
\lesssim_\epsilon 
N^\epsilon |\myset\cap[X]|^{r-1}
.\end{equation}
\end{theorem}

\end{document}